\documentclass[review]{elsarticle}
\usepackage[a4paper, left=2.5cm, right=2.5cm, top=2.5cm, bottom=2.5cm]{geometry}
\usepackage{hyperref}
\usepackage{amssymb}
\usepackage{amsmath}
\usepackage{amsthm}
\usepackage{graphicx}
\usepackage{dcolumn}
\usepackage{endnotes}
\usepackage{tabularx}
\usepackage{mathtools}
\usepackage{amsmath,amssymb,enumitem,amsfonts}
\usepackage{geometry}
\usepackage[matrix,arrow]{xy}

\journal{......}

\newtheorem{theorem}{Theorem}[section]
\newtheorem{proposition}[theorem]{Proposition}
\newtheorem{lemma}[theorem]{Lemma}
\newtheorem{corollary}[theorem]{Corollary}

\theoremstyle{definition}
\newtheorem{definition}[theorem]{Definition}
\newtheorem{example}[theorem]{Example}
\newtheorem{remark}[theorem]{Remark}
\newtheorem{question}[theorem]{Question}

\newcommand\twoheaduparrow{\mathord{\rotatebox[origin=c]{90}{$\twoheadrightarrow$}}}
\newcommand\twoheaddownarrow{\mathord{\rotatebox[origin=c]{90}{$\twoheadleftarrow$}}}
\newcommand{\dda}{\twoheaddownarrow}
\newcommand{\dua}{\twoheaduparrow}
\newcommand{\ua}{\mathord{\uparrow}}
\newcommand{\da}{\mathord{\downarrow}}

\newcommand{\mn}{\mathbb N}

\begin{document}
	
	\begin{frontmatter}
		
		\title{Convergence Choquet-complete spaces and domain representations}
		
		
		%
		%

		\author{Xiaoyong Xi}
		\address{School of Mathematics and Statistics, Yancheng Teachers University, Jiangsu, Yancheng, China}
		
		\author{Chong Shen\corref{cor1}}
		\ead{shenchong0520@163.com}
		\cortext[cor1]{Corresponding author.}                %
		\address{School of Science, Beijing University of Posts and Telecommunications, Beijing, China}
		
			\author{Weng Kin Ho}
		\address{Mathematics and Mathematical Education, National Institute of Education,
			Nanyang Technological University,  1 Nanyang Walk, Singapore}

		\author{Dongsheng Zhao}
		\address{Mathematics and Mathematical Education, National Institute of Education,
				Nanyang Technological University,  1 Nanyang Walk, Singapore}

\begin{abstract}
de Brecht, Goubault-Larrecq, Jia and Lyu asked whether every sober
convergence Choquet-complete space is domain-complete. We introduce the
notion of singleton Choquet-completeness, a weakening of convergence
Choquet-completeness in which the open sets chosen by player $\alpha$
are required to have a singleton intersection, but not necessarily to
form a neighbourhood basis. We prove that 
every singleton Choquet-complete $T_1$ space is domain-representable.
Consequently, every convergence Choquet-complete $T_1$ space is
domain-representable and hence sober. Thus, in the $T_1$ case, the
sobriety assumption in the above question is redundant, and the question
reduces to whether every convergence Choquet-complete $T_1$ space is
domain-complete.
\end{abstract}
		
		\begin{keyword}
			singleton Choquet-complete space, convergence Choquet-complete space, maximal filter space, Scott topology, domain representation, algebraic dcpo
			\MSC[2010] 06B35 \sep 06B30 \sep 54A05
		\end{keyword}
		
	\end{frontmatter}
	

	\section{Introduction}
	
Domain theory provides a powerful order-theoretic framework for topology.
The Scott topology on a dcpo plays a central role in domain theory and theoretical
computer science, but it is, in general, only $T_0$. However, by considering the
space $\operatorname{Max}(P)$ of maximal points of a poset $P$, equipped with the
relative Scott topology, one can obtain all $T_1$ spaces \cite{zhao,xi-zhao}. A
topological space $X$ is said to have a \emph{poset model} if it is homeomorphic
to such a maximal point space. In particular, spaces admitting a domain model are
called \emph{domain-representable} by Bennett and Lutzer
\cite{bennett,bennett1}, whereas Martin refers to them as spaces \emph{having a
	domain model} in several of his papers (e.g., \cite{martin}). A classical example
is the poset $\mathbb{IR}$ of all bounded closed intervals of real numbers,
ordered by reverse inclusion, which is a bounded complete $\omega$-domain model of the Euclidean space
$\mathbb{R}$.

Spaces with domain models have been extensively studied. Lawson proved that a
space is Polish if and only if it has an $\omega$-domain model satisfying the
Lawson condition \cite{Lawson97}. Martin \cite{ref3} showed that the maximal point
space $\operatorname{Max}(P)$ of an $\omega$-continuous dcpo $P$ is regular if and
only if it is Polish. Formal balls of metric spaces, introduced by Weihrauch and
Schreiber as a domain-theoretic representation of metric spaces
\cite{formal-1}, were later used by Edalat and Heckmann to prove that every
complete metric space is domain-representable \cite{formal-2}.

As Martin highlighted in \cite{ideal}, the property that $\operatorname{Max}(P)$ constitutes a $G_\delta$ subset is of great utility across a wide range of applications. A typical example can be found in Edalat's work \cite{edalat}, which explores the connections between measure theory and probabilistic powerdomains. His analysis hinges on the key assumption that each separable metric space can be embedded as a $G_\delta$-subspace within a countably based domain. Inspired by these application scenarios, Martin established two general results characterizing the $G_\delta$ property of spaces of maximal elements:
\begin{enumerate}
	\item Let $P$ be an $\omega$-domain. If $\operatorname{Max}(P)$ is regular, then $\operatorname{Max}(P)$ is a $G_\delta$ subset of $P$ \cite{ref3};
	\item Any $G_\delta$ subspace of a domain-representable space possesses an ideal domain model \cite{ideal};
	\item A $G_{\delta}$ subset of a Choquet complete space (particularly of a domain representable space) is Choquet complete \cite{ref3,hk}.
\end{enumerate}

More recently, de Brecht, Goubault-Larrecq, Jia, and Lyu \cite{LCS} studied
domain-complete spaces and LCS-complete spaces. In particular, they proved
that every continuous valuation on a $G_\delta$-subset of a locally compact
sober space extends to a measure. They also
introduced the weaker notion of \emph{convergence Choquet-completeness}, defined
via the strong Choquet game, and posed the following question.

\begin{question}
	Is every sober convergence Choquet-complete space domain-complete?
\end{question}

The main purpose of this paper is to show that, for $T_1$ spaces, the
sobriety assumption is unnecessary. To do this, we first introduce a
slight weakening of convergence Choquet-completeness, called
\emph{singleton Choquet-completeness}. Roughly speaking, this means that
player $\alpha$ has a winning strategy in the strong Choquet game such
that, for every play following this strategy, the open sets chosen by
$\alpha$ have a singleton intersection.

In fact, we prove the following slightly stronger result. For every
singleton Choquet-complete $T_1$ space $X$, we construct a poset $P_X$
such that $X$ is homeomorphic to $\operatorname{MF}(P_X)$, the space of
maximal filters on $P_X$. Our construction follows the approach to poset
spaces developed by Mummert and Stephan \cite{MummertS2010}: the basic
open sets of $X$ are encoded as conditions, and the order records finite
extensions of plays compatible with a fixed singleton winning strategy
for Player $\alpha$.

Finally, we show that the filter poset
$(\operatorname{Filt}(P_X),\subseteq)$ is an algebraic dcpo whose compact
elements are precisely the principal filters. Moreover, the topology on
$\operatorname{MF}(P_X)$ coincides with the subspace topology inherited from the
Scott topology on this algebraic dcpo. Consequently, $X$ admits a domain model.

\section{Preliminary}	
	The next part is devoted to a brief review of some basic concepts and notations that will be used later. For more details, we refer the readers to \cite{redbook,goubault}.
	
A \textit{poset} is a set $P$ equipped with a reflexive, antisymmetric, transitive binary relation $\leqslant$. 	For a subset $A$ of a poset $P$, we shall adopt the following standard notations:
	$$\ua A=\{y\in P:  \exists x\in A, x\leqslant y\};\ \da A=\{y\in P: \exists x\in A, y\leqslant x\}.$$
	For each $x\in X$, we simply write $\ua x$ and $\da x$ for $\ua\{x\}$ and  $\da \{x\}$, respectively. A subset $A$ of $P$ is called a \emph{lower} (resp., an \emph{upper}) \emph{set} if $A=\da A$ (resp., $A=\ua A$).
	An element $x$ is \emph{maximal}  in $A\subseteq P$, if  $A\cap \ua x=\{x\}$. The set of all maximal elements of $A$ is denoted by Max$(A)$. The set of minimal elements of $A$, denoted by Min$(A)$, is  defined dually. 
	
	A nonempty subset $D$ of $P$ is \emph{directed}  if every two
	elements in $D$ have an upper  bound in $D$.  A subset $X$ of $P$ is called an \emph{ideal} if $X$ is a directed lower set.   The set of all ideals of $P$ is denoted by Idl$(P)$.  A poset is said to be \emph{directed complete } (a \emph{dcpo} for short) if every directed subset has a sup.

	For  $x,y\in P$,  $x$ is \emph{way-below} $y$, denoted by $x\ll  y$, if for each directed subset
	$D$ of $P$ with  $\bigvee D$ existing,
	$y\leqslant\bigvee D$
	implies $x\leqslant d$ for some $d\in D$.  Denote $\dua x=\{y\in P:  x\ll y\}$ and $\dda x=\{y\in P:  y\ll x\}$. A poset $P$ is
	\emph{continuous}, if for each $x\in P$, the set $\dda x$ is directed and
	$x= \bigvee\dda x$.

	An element $a\in P$ is called \emph{compact}, if $a\ll a$. The set of all compact elements of $P$ will be denoted by   $K(P)$.
	Then, $P$ is called \emph{algebraic} if for each $x\in P$, the set $\{a\in K(P):a\leqslant x\}$ is directed and
	$x=\bigvee\{a\in K(P):  a\leqslant x\}$.
	A continuous (resp., algebraic) dcpo  is also called a \emph{domain} (resp., an \emph{algebraic domain}).
	A subset $B\subseteq P$ is a {\em base} of  $P$ if  for each $x\in P$, $B\cap \dda x$ is directed and $\bigvee(B\cap \dda x)$=$x$.

	\begin{remark}
		\begin{itemize}
			\item [(1)] A dcpo is a domain if and only if it has a base.
			\item [(2)] For each base $B$ of a domain $P$, $K(P)\subseteq B$. As a consequence, every algebraic domain $P$ has the smallest base, namely $K(P)$.
		\end{itemize}	
	\end{remark}
	
	A subset $U$ of  $P$ is \emph{Scott open} if
	(i) $U=\mathord{\uparrow}U$ and (ii) for each directed subset $D$ of $P$ for
	which $\bigvee D$ exists, $\bigvee D\in U$ implies $D\cap
	U\neq\emptyset$. All Scott open subsets of $P$ form a topology on $P$,
	called the \emph{Scott topology} and denoted by $\sigma(P)$. The space $\Sigma P=(P,\sigma(P))$ is called the
	\emph{Scott space} of $P$.
	
	The notion of ideal domain is introduced by K. Martin \cite{ideal}, defined as follows:
	\begin{definition}[\cite{ideal}]
		A continuous dcpo is \emph{ideal} if every element is either compact or maximal.
	\end{definition}

	The following lemma will be used in the sequel.
	
	\begin{lemma}[\cite{ideal}]\label{new-lem-2}
		An ideal domain is algebraic and first-countable in its Scott topology.
	\end{lemma}

	A \textit{filter} on a poset $P$ is a nonempty subset $F\subseteq P$ satisfying two axioms:
	\begin{enumerate}[label=(F\arabic*)]
		\item Downward directed: For every $p,q\in F$, there exists $r\in F$ such that $r\leqslant p$ and $r\leqslant q$;
		\item Upward closed: If $p\in F$ and $p\leqslant q$, then $q\in F$.
	\end{enumerate}

	A proper filter $F$ on $P$ is said to be an \emph{ultrafilter} (or a \emph{maximal filter}) if $F$ is maximal  among all proper filters on $P$; equivalently, there exists no proper filter properly containing $F$. Write $\operatorname{MF}(P)$ for the collection of all maximal filters on $P$. We topologize $\operatorname{MF}(P)$ using the basis $\{N_p :  p\in P\}$, where
	$$
	N_p = \{F\in \operatorname{MF}(P) : p\in F\}.
	$$

\begin{remark}\label{rem-1}
	Suppose $P$ is a poset, and let
	$$
	\operatorname{Filt}(P)=\{F\subseteq P: F \text{ is a filter on } P\},
	$$
	ordered by inclusion. Then $(\operatorname{Filt}(P),\subseteq)$ is an
	algebraic dcpo whose compact elements are precisely the principal filters
	$\uparrow p$, where $p\in P$. 
	Consequently, the Scott topology has a basis consisting of sets of the form
	$$
	\mathcal U_p
	=
	\{F\in\operatorname{Filt}(P):\uparrow p\subseteq F\}
	=
	\{F\in\operatorname{Filt}(P):p\in F\},
	$$
	where $p\in P$. Restricting this topology to the maximal points
	$\operatorname{MF}(P)\subseteq \operatorname{Filt}(P)$, we obtain
	$$
	\mathcal U_p\cap \operatorname{MF}(P)
	=
	\{F\in\operatorname{MF}(P):p\in F\}
	=
	N_p.
	$$
	Thus the topology on $\operatorname{MF}(P)$ given by the basis
	$\{N_p:p\in P\}$ is exactly the topology inherited from the Scott topology on
	the algebraic dcpo $(\operatorname{Filt}(P),\subseteq)$.
\end{remark}


	Let $X$ be a topological space. The \textit{strong Choquet game} on $X$ (see \cite{goubault}) is played by two players $\alpha$ and $\beta$, with $\beta$ moving first, following the rules below:
	\begin{enumerate}
		\item  Player $\beta$ picks a point $x_0\in X$ and an open set $V_0$ such that $x_0\in V_0$.
	Then player $\alpha$ responds by choosing an open set $U_0$ such that $x_0\in U_0$ and $U_0\subseteq V_0$.
		\item  Player $\beta$ selects a new point $x_1\in U_0$ and
		an open set $V_1$ such that $x_1\in V_1$ and $V_1\subseteq U_0$.
		\item The players alternate infinitely, producing a decreasing sequence of open sets
		$$
		V_0 \supseteq U_0 \supseteq V_1 \supseteq U_1 \supseteq V_2 \supseteq U_2 \supseteq \cdots
		$$
		satisfying 
		\begin{center}
			$x_n\in U_n$ and $x_n\in V_n$
		\end{center} for all $n\in\mathbb{N}$.
	\end{enumerate}
	
	An \textit{$\alpha$-history} is a finite sequence
	$$
	x_0,V_0,U_0,x_1,V_1,U_1,\dots,x_n,V_n
	$$
	obeying the above inclusion and membership rules.
	
	A \textit{strategy for $\alpha$} is a map $\sigma$ that takes any $\alpha$-history as input and outputs an open set $U_n$ such that $x_n\in U_n\subseteq V_n$. This map encodes how $\alpha$ reacts to every possible sequence of moves by $\beta$. Strategies are not required to be stationary: both players have full access to all prior points $x_k$ and open sets $U_k,V_k$ played in earlier rounds.

	\begin{enumerate}
		\item A space $X$ is \textit{Choquet-complete} (see \cite{goubault}) if and only if player $\alpha$ possesses a winning strategy, meaning that whatever  $\beta$ plays, $\alpha$ always has a way of playing such that 
		$\bigcap_{n\in\mathbb{N}} U_n \left(= \bigcap_{n\in\mathbb{N}} V_n \right) \neq \emptyset$.
		\item A space $X$ is \textit{convergence Choquet-complete} (see \cite{CCh-complete}) if $\alpha$ has a  winning strategy such that the sequence $(U_n)_{n\in\mathbb{N}}$ forms a neighborhood basis of open sets for some single point of $X$.
	\end{enumerate}

\begin{remark}\label{rem:convergence-first-countable}
	Every convergence Choquet-complete $T_1$ space is first-countable.
	Indeed, let $s_\alpha$ be a convergent winning strategy for player
	$\alpha$, and fix $x\in X$. Let player $\beta$ play the same point $x$ at
	each inning, together with legal open neighbourhoods of $x$. Then every
	open set $U_n$ chosen by player $\alpha$ contains $x$. Since
	$s_\alpha$ is convergent, the sequence $(U_n)_{n\in\mathbb N}$ forms a
	neighbourhood basis at some point, say $y\in X$.
	
	We claim that $y=x$. If $y\neq x$, then, since $X$ is $T_1$, there exists
	an open neighbourhood $O$ of $y$ such that $x\notin O$. As $(U_n)$ is a
	neighbourhood basis at $y$, there exists $n$ such that $U_n\subseteq O$.
	But $x\in U_n$, a contradiction. Hence $y=x$. Therefore $(U_n)$ is a
	countable neighbourhood basis at $x$. Since $x$ was arbitrary, $X$ is
	first-countable.
\end{remark}

	\section{Main results}
	
	We first introduce a slight weakening of convergence Choquet-completeness.
	
	\begin{definition}
		A space $X$ is called \textit{singleton Choquet-complete} if player
		$\alpha$ has a strategy in the strong Choquet game such that, for every
		play following this strategy, there exists a point $x\in X$ with
		\[
		\bigcap_{n\in\mathbb N}U_n=\{x\},
		\]
		where $U_n$ denotes the open set chosen by player $\alpha$ at the $n$-th
		inning.
	\end{definition}
	
	It is clear that every convergence Choquet-complete $T_1$ space is
	singleton Choquet-complete. The converse, however, is not true, as the following example shows.
\begin{example}\label{ex:join-topology}
	Let
	\[
	\tau=\tau_{\mathrm{usual}}\vee\tau_{\mathrm{cocountable}}
	\]
	be the join of the usual topology and the cocountable topology on
	$\mathbb R$. Then $(\mathbb R,\tau)$ is singleton Choquet-complete, but
	not convergence Choquet-complete.
	
	First, we show that $(\mathbb R,\tau)$ is singleton Choquet-complete.
	A basic open neighbourhood of a point $x\in\mathbb R$ is of the form
	$U\setminus C$, where $U$ is a usual open neighbourhood of $x$ and $C$ is
	countable with $x\notin C$.
	
Suppose player $\beta$ plays $x_n\in V_n$ at the $n$-th inning. Choose a
basic neighbourhood $O_n\setminus C_n$ of $x_n$ such that
$x_n\in O_n\setminus C_n\subseteq V_n$. We describe a strategy for
player $\alpha$. During the construction, player $\alpha$ also keeps a
finite set $F_n\subseteq\bigcup_{m\leq n}C_m$ such that every point of
$\bigcup_{n\in\mathbb N}C_n$ eventually belongs to some $F_n$.
For instance, one may enumerate the points of the countable union as they
appear and let $F_n$ consist of the first finitely many points in this list.
Player $\alpha$ chooses a usual open interval $I_n$ such that
$x_n\in I_n\subseteq O_n$,
$\overline{I_n}\subseteq I_{n-1}$ if $n>0$,
$\operatorname{diam}(I_n)<2^{-n}$, and
$\overline{I_n}\cap F_n=\emptyset$. This is possible because $F_n$ is
finite and $x_n$ belongs to none of the sets $C_0,\ldots,C_n$.
Indeed, if $m<n$, then $x_n\in V_n\subseteq U_{n-1}\subseteq \mathbb R\setminus C_m$,
while $x_n\notin C_n$ by the choice of $O_n\setminus C_n$.
Player $\alpha$ then plays
\[
U_n=I_n\setminus\bigcup_{m\leq n}C_m.
\]
This is a legal move.

The closed intervals $\overline{I_n}$ are nested and have diameters
tending to $0$. Hence
\[
\bigcap_{n\in\mathbb N}\overline{I_n}=\{p\}
\]
for some $p\in\mathbb R$. Since every point of
$\bigcup_{n\in\mathbb N}C_n$ eventually belongs to some $F_n$, and the
later intervals are nested inside the earlier ones, the condition
$\overline{I_n}\cap F_n=\emptyset$ ensures that
$p\notin\bigcup_{n\in\mathbb N}C_n$. Therefore
\[
p\in\bigcap_{n\in\mathbb N}U_n.
\]
Conversely, any point in $\bigcap_{n\in\mathbb N}U_n$ belongs to
$\bigcap_{n\in\mathbb N}\overline{I_n}=\{p\}$. Thus
\[
\bigcap_{n\in\mathbb N}U_n=\{p\}.
\]
So $(\mathbb R,\tau)$ is singleton Choquet-complete.

We now show that $(\mathbb R,\tau)$ is not convergence Choquet-complete.
Indeed, it is easy to see that $(\mathbb R,\tau)$ is not first-countable.
By Remark~\ref{rem:convergence-first-countable}, every convergence
Choquet-complete $T_1$ space is first-countable. Therefore
$(\mathbb R,\tau)$ is not convergence Choquet-complete.
	
\end{example}

Let $X$ be a singleton Choquet-complete $T_1$ space. Fix a basis
$\mathcal B$ for the topology of $X$, and fix a singleton winning
strategy $s_{\alpha}$ for player $\alpha$ in the strong Choquet game on
$X$. We use the poset of conditions constructed by Mummert and Stephan
\cite{MummertS2010}, denoted here by $P_X$. For completeness, we recall
its construction in the present setting.
	
	A condition $c$ is a finite list of the form
	$$
	\langle A,\pi_1,\pi_2,\dots,\pi_k\rangle
	$$
	satisfying the following requirements:
	\begin{enumerate}[label=(C\arabic*)]
		\item $A$ is a nonempty member of $\mathcal B$. For any condition $c$,
		write $S(c)=A$ for the basic open set appearing as the first entry of $c$.
		\item The empty play $\langle\rangle$ belongs to $c$.
		\item Each $\pi_i$ is a finite partial play of the strong Choquet game on
		$X$ following the fixed strategy $s_\alpha$, and each such play terminates
		with a move by player $\alpha$. We write $U(\pi_i)$ for its last open set.
		For the empty play, put $U(\langle\rangle)=X$.
		\item If a play $\pi$ belongs to $c$, then every initial segment of $\pi$
		ending with a move by player $\alpha$ also belongs to $c$.
		\item $A\subseteq U(\pi_i)$ for every play $\pi_i$ occurring in $c$.
	\end{enumerate}
	
	Let $c=\langle A,\pi_1,\dots,\pi_k\rangle$ and $c'=\langle A',\pi_1',\dots,\pi_l'\rangle$ be two conditions. We define $c'<c$ if and only if the following two statements hold:
	\begin{enumerate}[label=(P\arabic*)]
		\item For every finite play $\pi_i$ contained in $c$, there exists a point $x\in S(c)=A$ such that the extended play
		$$
		\pi_i \frown \langle A,x,s_{\alpha}\big(\pi_i \frown \langle A,x\rangle\big)\rangle
		$$
		appears as some $\pi_j'$ in $c'$.
		\item $A'\subseteq A$.
		This condition ensures that stronger conditions have smaller associated basic
		open sets.
	\end{enumerate}
	
	The relation $\leqslant$ defined by
	$$
	c'\leqslant c \quad \Longleftrightarrow \quad c'<c \text{ or } c'=c
	$$
	is a partial order on $P_X$. Thus stronger conditions are smaller in the
	order. 
	
	\begin{remark}
		It is worth noting that
	\begin{itemize}
		\item $S:P_Xo \mathcal O(X)$ is order-preserving, that is, $c\leqslant c'$ implies $S(c)\subseteq S(c')$. 
	\end{itemize}	
	\end{remark}

	\begin{lemma}\label{lem-1}
		Suppose that $X$ is singleton Choquet-complete. For any decreasing countable chain 
		$$c_0>c_1>c_2>c_3>\cdots>c_n>c_{n+1}>\cdots$$
		in $P_X$, there exists a point $x\in X$ such that
		$$
		\bigcap_{i\in\mn}S(c_i)=\{x\}.
		$$
	\end{lemma}
\begin{proof}
	It is clear that 
$$S(c_0) \supseteq S(c_1)\supseteq S(c_2) \supseteq \cdots \supseteq S(c_n) \supseteq S(c_{n+1}) \supseteq \cdots.$$
	
	We inductively construct a sequence of finite plays $\langle \pi_i :  i \in \mathbb{N} \rangle$ such that $\pi_{i+1}$ is an immediate extension of $\pi_i$ for all $i \in \mathbb{N}$.
	\begin{itemize}
		\item At stage $0$, pick an arbitrary finite play $\pi_0$ contained in $c_0$.
		\item At stage $i + 1$, as $c_{i+1}< c_i$, one can choose $\pi_{i+1}$ to be a finite play in $c_{i+1}$ which is an immediate extension of $\pi_i$. Explicitly, since $c_{i+1}<c_i$, there exists a point $x\in S(c_i)$ such that
		$$
		\pi_{i+1} := \pi_i \frown \big\langle S(c_i),x,s_{\alpha}\big(\pi_i \frown \langle S(c_i),x\rangle\big)\big\rangle \in c_{i+1}.
		$$
	\end{itemize}
	Once the sequence $\langle \pi_i \rangle_{i\in\mathbb{N}}$ is constructed,
	these increasing finite plays determine an infinite play $\gamma$ of the
	strong Choquet game following the fixed strategy $s_\alpha$. Since
	$s_\alpha$ is a singleton winning strategy, there exists $x\in X$ such
	that
	$$
	\bigcap_{i\in\mathbb N}U(\pi_i)=\{x\}.
	$$
	Moreover, if $m\leqslant j$, then $\pi_m$ is an initial segment of
	$\pi_j$, hence $\pi_m$ occurs in $c_j$. Therefore
	$S(c_j)\subseteq U(\pi_m)$. It follows that
	$$
	\bigcap_{i\in\mathbb N}S(c_i)\subseteq
	\bigcap_{i\in\mathbb N}U(\pi_i)=\{x\}.
	$$
	On the other hand, since each $S(c_i)$ is one of the open sets played by
	player $\beta$ along $\gamma$, and since $x$ belongs to all open sets in
	the play, we have $x\in S(c_i)$ for every $i\in\mathbb N$. Hence
	$\bigcap_{i\in\mathbb N}S(c_i)=\{x\}$.
\end{proof}

	\begin{lemma}\label{lem-2}
Let $P$ be a poset, let $X$ be a $T_1$ space, and let
$S:Po(\mathcal O(X),\subseteq)$ be order-preserving. Suppose that $F$ is a
filter on $P$ with no smallest element and that, for every strictly decreasing
chain
$$
c_0>c_1>c_2>\cdots
$$
contained in $F$, there exists a point $x_{(c_i)}\in X$ such that
$$
\bigcap_{i\in\mn}S(c_i)=\{x_{(c_i)}\}.
$$
Then there exists a unique point $x\in X$ such
that
$$
\bigcap_{c\in F} S(c)=\{x\}.
$$
	\end{lemma}
	
	\begin{proof}
First observe that for every $a\in F$ there exists $b\in F$ such that
$b<a$. Indeed, since $a$ is not the smallest element of $F$, there is
$q\in F$ such that $a\not\leqslant q$. Since $F$ is filtered, choose
$b\in F$ with $b\leqslant a$ and $b\leqslant q$. Then $b\neq a$, and hence
$b<a$.

Choose a strictly decreasing chain
$$
c_0>c_1>c_2>\cdots
$$
in $F$. By assumption, there exists $x\in X$ such that
$$
\bigcap_{i\in\mn}S(c_i)=\{x\}.
$$

Let
$$
\widetilde c_0>\widetilde c_1>\widetilde c_2>\cdots
$$
be another strictly decreasing chain in $F$, and suppose that
there exists $y\in X$ such that
$$
\bigcap_{i\in\mn}S(\widetilde c_i)=\{y\}.
$$
We claim that $x=y$.

Construct a strictly decreasing chain
$$
d_0>d_1>d_2>\cdots
$$
in $F$ such that $d_i\leqslant c_i$ and $d_i\leqslant \widetilde c_i$ for all
$i$. This is done inductively. First choose a common lower bound $d_0$ of $c_0$
and $\widetilde c_0$ in $F$. At stage $i+1$, choose a common lower bound of
$d_i,c_{i+1},\widetilde c_{i+1}$ in $F$, and then choose a strictly smaller
element $d_{i+1}$ of $F$.

By assumption, there exists $z\in X$ such that
$$
\bigcap_{i\in\mn}S(d_i)=\{z\}.
$$
Since $S$ is order-preserving,
$$
\{z\}
=
\bigcap_{i\in\mathbb N} S(d_i)
\subseteq
\bigcap_{i\in\mathbb N} S(c_i)\cap \bigcap_{i\in\mathbb N} S(\widetilde c_i)
=
\{x\}\cap\{y\}.
$$
Hence $x=y$.

Finally, let $c\in F$. By the first paragraph, $c$ can be extended to a
strictly decreasing chain in $F$ starting with $c$. Its associated point must
be the same point $x$, by the uniqueness just proved. Therefore $x\in S(c)$.
Since $c\in F$ was arbitrary,
$$
x\in\bigcap_{c\in F}S(c).
$$
Conversely, since the chain $(c_i)_{i\in\mn}$ chosen above is contained in
$F$,
$$
\bigcap_{c\in F}S(c)\subseteq \bigcap_{i\in\mathbb N} S(c_i)=\{x\}.
$$
Therefore,
$
\bigcap_{c\in F}S(c)=\{x\}$.
	\end{proof}

\begin{lemma}\label{lem-3}
	Let $c_1,c_2\in P_X$ and let
	$$
	x\in S(c_1)\cap S(c_2).
	$$
	Then there exists a condition $c\in P_X$ such that
	$$
	c\prec c_1,\qquad c\prec c_2,
	\qquad x\in S(c).
	$$
\end{lemma}

\begin{proof}
		For every finite play $\pi$ occurring in $c_i$, extend $\pi$ by
	letting Player $\beta$ play the point $x$ and the open set $A_i$,
	and then letting Player $\alpha$ respond according to the fixed
	strategy $s_\alpha$. Thus, form the play
	$$
	\widehat{\pi}
	=
	\pi\frown
	\left\langle
	A_i,x,
	s_\alpha\bigl(\pi\frown\langle A_i,x\rangle\bigr)
	\right\rangle .
	$$
	This is a legal extension because
	$$
	x\in A_i=S(c_i)\subseteq U(\pi).
	$$
	
	Let $\mathcal P$ consist of all such extended plays
	$\widehat{\pi}$, together with all their initial segments ending
	with a move by Player $\alpha$. Since $c_1$ and $c_2$ contain only
	finitely many plays, $\mathcal P$ is finite.
	
	Every $U(\rho)$, with $\rho\in\mathcal P$, is an open
	neighborhood of $x$. Hence
	$$
	W=\bigcap_{\rho\in\mathcal P}U(\rho)
	$$
	is an open neighborhood of $x$. Choose a nonempty basic open set
	$A$ such that
	$$
	x\in A\subseteq W\cap A_1\cap A_2.
	$$
	Define
	$$
	c=\langle A,\mathcal P\rangle.
	$$
	By construction, $c$ satisfies all the requirements of a condition
	in $P_X$, and $S(c)=A$ contains $x$.
	
	Moreover, every play occurring in $c_1$ has an immediate extension
	occurring in $c$, and the same holds for every play occurring in
	$c_2$. Therefore,  $c< c_1$
and
	$c< c_2$.
\end{proof}

\begin{proposition}\label{prop-1}
	Suppose that $X$ is a singleton Choquet-complete $T_1$ space.
	If $F$ is a filter on $P_X$ with no smallest element, then there exists
	a unique point $x\in X$ such that
	$$
	\bigcap_{c\in F}S(c)=\{x\}.
	$$
	In particular, the same conclusion holds for every maximal filter on $P_X$.
\end{proposition}

\begin{proof}
	The first assertion follows from Lemmas~\ref{lem-1} and~\ref{lem-2}.
	
	It remains to show that every maximal filter on $P_X$ has no smallest
	element. Suppose, to the contrary, that a maximal filter $F$ has a smallest
	element $c_0$. Then $F=\uparrow c_0$. Choose $x\in S(c_0)$. By
	Lemma~\ref{lem-3}, there exists $d\in P_X$ such that $d < c_0$ and
	$x\in S(d)$. Hence
	$$
	\uparrow c_0\subsetneq \uparrow d,
	$$
	contradicting the maximality of $F$. Therefore every maximal filter on
	$P_X$ has no smallest element.
\end{proof}

\begin{lemma}\label{lem-4}
	Suppose that $X$ is a singleton Choquet-complete $T_1$ space.
	For every $x\in X$, define 
	$$
	F_x=\{c\in P_X:x\in S(c)\}.
	$$
Then $F_x$	is a maximal filter on $P_X$ such that
 $$\bigcap_{c\in F_x}S(c)  = \{x\}.$$
\end{lemma}

\begin{proof}
(1) By Lemma~\ref{lem-3} and the
	order-preserving property of $S$, the set $F_x$ is a filter on
	$P_X$.
	
	We now show that $\bigcap_{c\in F_x}S(c)=\{x\}$. By the
	definition of $F_x$, we have $x\in S(c)$ for every $c\in F_x$,
	and hence $x\in\bigcap_{c\in F_x}S(c)$.
	
	Now let $y\neq x$. Since $X$ is a $T_1$ space, there exists a
	basic open set $A$ such that $x\in A$ and $y\notin A$. Then
	$c_A=\langle A,\langle\rangle\rangle$ is a condition in $P_X$.
	Moreover, $x\in S(c_A)=A$, so $c_A\in F_x$. Since
	$y\notin S(c_A)$, we have
	$y\notin\bigcap_{c\in F_x}S(c)$. Therefore,
	$$
	\bigcap_{c\in F_x}S(c)=\{x\}.
	$$
	
	\medskip
	
(2)	To show that $F_x$ is maximal, let $G$ be a filter on $P_X$ such that
	$F_x\subseteq G$. We first show that $G$ has no smallest element.
	
	Suppose, to the contrary, that $g_0$ is the smallest element of $G$.
	Then $g_0\leqslant c$ for every $c\in F_x$. Since $S$ is
	order-preserving,
	$$
	S(g_0)\subseteq \bigcap_{c\in F_x}S(c)=\{x\}.
	$$
	As $S(g_0)$ is nonempty, this implies $S(g_0)=\{x\}$. Thus
	$x\in S(g_0)$. By Lemma~\ref{lem-3}, there exists $d\in P_X$ such that
	$d\prec g_0$ and $x\in S(d)$. Hence $d\in F_x\subseteq G$, which
	contradicts the fact that $g_0$ is the smallest element of $G$. Therefore
	$G$ has no smallest element.
	
	By Proposition~\ref{prop-1}, there exists $z\in X$ such that
	$$
	\bigcap_{c\in G}S(c)=\{z\}.
	$$
	Since $F_x\subseteq G$, we have
	$$
	\{z\}=\bigcap_{c\in G}S(c)
	\subseteq
	\bigcap_{c\in F_x}S(c)=\{x\}.
	$$
	Thus $z=x$. Hence $x\in S(c)$ for every $c\in G$, so
	$G\subseteq F_x$. Since $F_x\subseteq G$, we conclude that $G=F_x$.
	Therefore $F_x$ is maximal.
\end{proof}

With the above preparations, we give our main result in this paper.

\begin{theorem}\label{thm-representation}
	Let $X$ be a singleton Choquet-complete $T_1$ space. Then $X$ is
	homeomorphic to $\operatorname{MF}(P_X)$.
\end{theorem}

\begin{proof}
	For each $x\in X$, put
	$$
	F_x=\{c\in P_X:x\in S(c)\},
	$$
	and define $\phi\colon Xo\operatorname{MF}(P_X)$ by
	$\phi(x)=F_x$. 
Next, we prove that $\phi$ is a homeomorphism in some steps.	
	
	{\bf Step 1: } $\phi$ is a bijection.
	
	By Lemma~\ref{lem-4}, $F_x$ is a maximal filter, so
	$\phi$ is well-defined.
		Conversely, for each $F\in\operatorname{MF}(P_X)$,
	Proposition~\ref{prop-1} yields a unique point $\varphi(F)\in X$
	such that
	$$
	\bigcap_{c\in F}S(c)=\{\varphi(F)\}.
	$$
	This defines a map
	$\varphi\colon\operatorname{MF}(P_X)\longrightarrow X$.
	
	We show that $\phi$ and $\varphi$ are inverse to each other. For
	$x\in X$, Lemma~\ref{lem-4} gives
	$$
	\bigcap_{c\in F_x}S(c)=\{x\},
	$$
	and hence $\varphi(\phi(x))=x$.
	Now let $F\in\operatorname{MF}(P_X)$ and put $x=\varphi(F)$. Then
	$x\in S(c)$ for every $c\in F$, so $F\subseteq F_x$. Since both
	$F$ and $F_x$ are maximal filters, it follows that $F=F_x$.
	Therefore $\phi(\varphi(F))=F$. Thus $\phi$ is a bijection whose
	inverse is $\varphi$. Therefore, $\phi$ is a bijection.
	
{\bf Step 2: }	$\phi$ is a continuous map.

 Recall
	that the sets
	$$
	N_c=\{F\in\operatorname{MF}(P_X):c\in F\},
	\qquad c\in P_X,
	$$
	form a basis for the topology on $\operatorname{MF}(P_X)$. For
	every $c\in P_X$, we have
	$$
	\phi^{-1}(N_c)
	=\{x\in X:c\in F_x\}
	=\{x\in X:x\in S(c)\}
	=S(c).
	$$
	Since $S(c)$ is open in $X$, the map $\phi$ is continuous.

{\bf Step 3: } $\phi$ is an open map. 

Let $A$ be a basic
open subset of $X$, and define
$$
c_A=\langle A,\langle\rangle\rangle.
$$
Then $c_A\in P_X$ and $S(c_A)=A$. We claim that
$\phi(A)=N_{c_A}$. Indeed, if $x\in A=S(c_A)$, then $c_A\in F_x$, and hence
$\phi(x)=F_x\in N_{c_A}$. Thus $\phi(A)\subseteq N_{c_A}$. Conversely, let $F\in N_{c_A}$ and put $x=\varphi(F)$.  Since
$c_A\in F$, we have
$$
x\in\bigcap_{c\in F}S(c)\subseteq S(c_A)=A.
$$
Note that $F=F_x=\phi(x)$, because $\phi$ and $\varphi$ are inverse
maps. Hence $F=\phi(x)\in\phi(A)$, proving that
$N_{c_A}\subseteq\phi(A)$. Therefore,
$$
\phi(A)=N_{c_A},
$$
which is open in $\operatorname{MF}(P_X)$.
This shows that $\phi$ is an open map.

Therefore, $\phi$ is a homeomorphism.
\end{proof}

Combining Theorem~\ref{thm-representation} with Remark~\ref{rem-1}, we obtain the following consequence.

\begin{corollary}
	Every singleton Choquet-complete $T_1$ space admits a domain model.
\end{corollary}
 Moreover, by the standard fact that maximal point spaces
of domains are sober, we obtain the following consequence.

\begin{theorem}
	Every singleton Choquet-complete $T_1$ space is sober.
\end{theorem}

\begin{remark}
	By Theorem~\ref{thm-representation}, the space
	$(\mathbb R,\tau_{\mathrm{usual}}\vee\tau_{\mathrm{cocountable}})$ in
	Example~\ref{ex:join-topology} has a domain model, although it is not
	convergence Choquet-complete. Thus Example~\ref{ex:join-topology}
	shows that the first implication below is not reversible:
	\[
	\text{convergence Choquet-complete}
	\Longrightarrow
	\text{singleton Choquet-complete}
	\Longrightarrow
	\text{domain-representable}.
	\]
	
	Consequently, in the $T_1$ case, the sobriety assumption in the
	question of de Brecht, Goubault-Larrecq, Jia and Lyu
	\cite[p.~33, Question~(iii)]{LCS} is unnecessary: every convergence
	Choquet-complete $T_1$ space admits a domain representation. Hence, in
	the $T_1$ setting, their question reduces to whether every convergence
	Choquet-complete $T_1$ space is domain-complete. This question remains
	open.
\end{remark}

We conclude this paper with a concrete example in the last section.
\begin{example}
	Let $\mathbb{R}$ be the real line with the usual topology. Then
	$\mathbb R$ has an $\omega$-ideal model $P$; see Example~4.3 of
	\cite{ideal}. Thus $\operatorname{Max}(P)$ is homeomorphic to
	$\mathbb R$.
	
	Let
	$$
	P^*=P\cup\{\omega\},
	$$
	where $\omega\notin P$. We keep the original order on $P$ and add the
	new order relations
	$$
	x\leqslant\omega \quad\text{whenever}\quad
	x\leqslant r
	\text{ for some } r\in\mathbb R\setminus\mathbb Q.
	$$
	Then one checks that $P^*$ is an $\omega$-algebraic dcpo such that
	$$
	\operatorname{Max}(P^*)=\mathbb Q\cup\{\omega\}
	\quad\text{and}\quad
	K(P^*)=K(P)\cup\{\omega\}.
	$$
	Moreover, $\mathbb Q$ is an open-and-closed subspace of
	$\operatorname{Max}(P^*)$.
	
Since $\mathbb Q$ is not Choquet-complete, and every domain-representable
space is Choquet-complete, $\mathbb Q$ is not domain-representable.
Therefore a domain-representable space may have an open-and-closed subspace
which is not domain-representable. This provides an alternative negative
answer to Martin's question \cite{ideal} on whether domain-representability
is hereditary to closed subspaces; see also the answer given by Bennett and
Lutzer in \cite{BL-domain}.
	
	We now show that $\operatorname{Max}(P^*)=\mathbb Q\cup\{\omega\}$ is
	not convergence Choquet-complete, and hence is not domain-complete.
	Let $\mathbb Q=\{r_n:n\in\mathbb N\}$. Suppose player $\alpha$ is given
	an arbitrary strategy. Player $\beta$ starts by playing $r_0$ and a
	basic open neighbourhood of $r_0$ of the form
	$\ua k_0\cap\operatorname{Max}(P^*)$, where $k_0\in K(P^*)$ and
	$r_0\in\ua k_0$. After $\alpha$ responds, player $\beta$ chooses
	$$
	r_{n_1},\quad
	n_1=\min\{n\in\mathbb N:
	r_n \text{ belongs to the open set just played by } \alpha\},
	$$
	and then plays a basic open neighbourhood
	$\ua k_1\cap\operatorname{Max}(P^*)$ of $r_{n_1}$ contained in
	$\alpha$'s previous move. Continuing in this way, we obtain a play
	$$
	r_{n_0},V_0,\ua k_0\cap\operatorname{Max}(P^*),
	r_{n_1},V_1,\ua k_1\cap\operatorname{Max}(P^*),\ldots
	$$
	following the strategy of $\alpha$, where the rational points
	$r_{n_i}$ are chosen successively so that
	$$
	\mathbb Q\cap\bigcap_{i\in\mathbb N}\ua k_i=\emptyset.
	$$
	On the other hand, each set $\ua k_i\cap\mathbb R$ is a nonempty open
	subset of $\mathbb R$, and hence contains an irrational number. By the
	definition of the order on $P^*$, this implies that
	$\omega\in\ua k_i$ for every $i\in\mathbb N$. Therefore
	$$
	\bigcap_{i\in\mathbb N}
	(\ua k_i\cap\operatorname{Max}(P^*))=\{\omega\}.
	$$
	However, $\{\omega\}$ is open in $\operatorname{Max}(P^*)$, while none
	of the open sets $\ua k_i\cap\operatorname{Max}(P^*)$ is contained in
	$\{\omega\}$. Hence these open sets cannot form a neighbourhood basis
	at $\omega$. Thus $\operatorname{Max}(P^*)$ is not convergence
	Choquet-complete.
	
\end{example}

\end{document}